\documentclass[oneside,10pt]{amsart}
\usepackage{amsmath}
\usepackage{amssymb}
\usepackage{amsthm}
\usepackage{mathtools}
\usepackage{float}
\usepackage{enumerate}
\usepackage{ulem}
\usepackage{xcolor}
\usepackage[utf8x]{inputenc}
\usepackage{comment}
\usepackage{framed}
\usepackage{caption}

\DeclarePairedDelimiter{\norm}{\lVert}{\rVert}

\DeclareMathAlphabet{\mathbbb}{U}{bbold}{m}{n}

\newtheorem{thm}{Theorem}[section]
\newtheorem{cor}[thm]{Corollary}

\theoremstyle{definition}

\theoremstyle{remark}
\newtheorem{rem}[thm]{Remark}

\numberwithin{equation}{section}
\usepackage{graphicx}

\usepackage{enumerate}

\newcommand{\R}{{\mathbb R}}

\newcommand{\matrice}{\begin{pmatrix}}
\newcommand{\ok}{\end{pmatrix}}

\newcommand\dom{\operatorname{dom}}

\DeclareMathOperator{\Div}{div}
\DeclareMathOperator{\curl}{curl}

\newcommand\p\partial

\newcommand{\I}{{\rm i}}
\newcommand{\la}{\lambda}

\begin{document}
\title[]{A note on the Maxwell's eigenvalues on thin sets}
\thanks{The first author acknowledges support of the INdAM-GNAMPA. The second author acknowledges the support of the INdAM GNSAGA. The second author also aknowledges financial support from the project ``Perturbation problems and asymptotics for elliptic differential equations: variational and potential theoretic methods'' funded by the European Union – Next Generation EU and by MUR-PRIN-2022SENJZ3 and from the project ``Analisi Geometrica e Teoria Spettrale su varietà Riemanniane ed Hermitiane'' of the INdAM GNSAGA}

\author[Provenzano]{Luigi Provenzano}
\address{Dipartimento di Scienze di Base e Applicate per l'Ingegneria, Universit\`a di Roma ``La Sapienza'', Via Scarpa 12 - 00161 Roma, Italy, e-mail: {\sf luigi.provenzano@uniroma1.it}.}
\author[Ferraresso]{Francesco Ferraresso}
\address{Dipartimento di Informatica, Università di Verona, Strada Le Grazie 15, 37134 Verona, Italy, e-mail: {\sf francesco.ferraresso@univr.it}.}

\begin{abstract}
We analyse the Maxwell's spectrum on thin tubular neighborhoods of embedded surfaces of $\mathbb R^3$. We show that the Maxwell's eigenvalues converge to the Laplacian eigenvalues of the surface as the thin parameter tends to zero. To achieve this, we reformulate the problem in terms of the spectrum of the Hodge Laplacian with relative conditions acting on co-closed differential $1$-forms. The result leads to new examples of domains where the Faber-Krahn inequality for Maxwell's eigenvalues fails, examples of domains with any number of arbitrarily small eigenvalues, and underlines the failure of spectral stability under singular perturbations changing the topology of the domain. Additionally, we explicitly produce Maxwell's eigenfunctions on product domains with the product metric, extending previous constructions valid in the Euclidean case.
\end{abstract}


\maketitle

\section{Introduction and statement of the main results} 
Let $\Omega$ be a bounded  domain in $\mathbb R^3$. The second-order reformulation of the time-harmonic Maxwell's system 
\[
\begin{cases}
\curl E = \I \eta H, \quad &\textup{in $\Omega$,} \\
\curl H = -\I \eta E, \quad &\textup{in $\Omega$}, \\
\nu \times E = 0, \quad &\textup{on $\p \Omega$}
\end{cases}
\]
is given by
\begin{equation} \label{eq:curlcurl}
\begin{cases}
\curl\curl E = \la E, \quad &\textup{in $\Omega,$} \\
\nu \times E = 0, \quad &\textup{on $\p \Omega,$}
\end{cases}
\end{equation}
where $\nu$ is the outer normal to $\p \Omega$ and $\la := \eta^2$ is the eigenvalue.  If $\Omega$ is sufficiently regular, e.g., if $\p \Omega$ is Lipschitz, it is well-known that problem \eqref{eq:curlcurl} has $\la = 0$ as eigenvalue of infinite multiplicity and it further admits a sequence of non-negative eigenvalues of finite multiplicity (corresponding to divergence-free eigenfunctions)
\[
0 \leq \la_1(\Omega) \leq \la_2(\Omega) \leq \dots \leq \la_j(\Omega) \leq \dots\nearrow+\infty,
\]
where the eigenvalues are repeated according to their multiplicity.

We are mainly interested in the dependence of $\la_j(\Omega)$ upon the perturbation of the domain $\Omega$; in particular, we will assume that $\Omega$ is a thin domain, described by tubes of size $h$ around smooth embedded surfaces (with or without boundary). More precisely, if $\Sigma$ is a smooth, embedded, orientable, compact surface in $\mathbb R^3$ (with or without boundary), we define, for all $h>0$ sufficiently small, the tube $\Omega_h$ by
\begin{equation}\label{tube}
\Omega_h:=\{x+t\nu(x):t\in(0,h),x\in\Sigma\},
\end{equation}
where $\nu$ is a choice of a unit normal vector field on $\Sigma$, and $\nu(x)$ is the corresponding unit normal vector at $x\in\Sigma$. Note that if $\Sigma$ has a boundary, $\Omega_h$ is just a piecewise smooth, Lipschitz domain. See Figures \ref{fig1} and \ref{fig2}.

\begin{figure}
\includegraphics[width=0.9\textwidth]{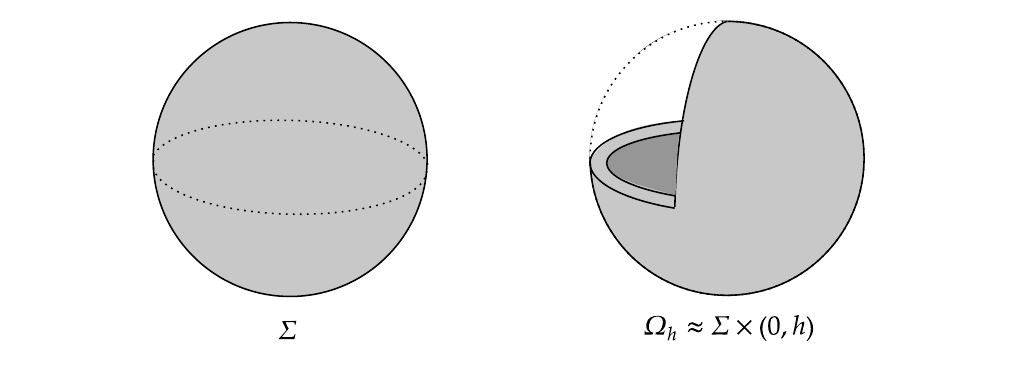}
\caption{Surface $\Sigma$ without boundary and domain $\Omega_h$}
\label{fig1}
\end{figure}

\begin{figure}
\includegraphics[width=0.9\textwidth]{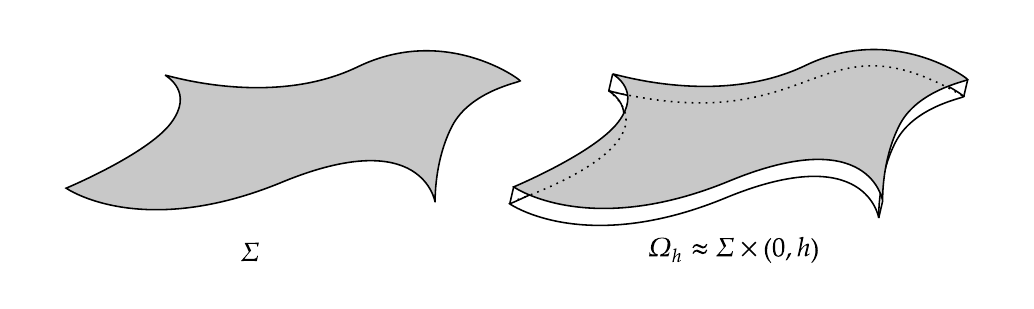}
\caption{Surface $\Sigma$ with boundary and domain $\Omega_h$}
\label{fig2}
\end{figure}

If the boundary is just Lipschitz, problem \eqref{eq:curlcurl}, and in particular, the boundary condition $\nu \times E|_{\p \Omega_h} = 0$, has to be interpreted in a suitable weak sense, see \cite{bu_co_sh}. We now state our main result.

\begin{thm}\label{main_convergence}
Let $\Sigma$ be a smooth, compact, embedded, orientable surface in $\mathbb R^3$, and let $\Omega_h$ be the tube of size $h$ around $\Sigma$ defined by \eqref{tube}. Let $\{\lambda_j(\Omega_h)\}_{j=1}^{\infty}$ be the sequence of Maxwell's eigenvalues. Then, for all $j\in\mathbb N$:
\begin{enumerate}[i)]
\item if $\partial\Sigma=\emptyset$, $\lim_{h\to 0^+}\lambda_j(\Omega_h)=\mu_j$, where $\{\mu_j\}_{j=1}^{\infty}$ are the Laplacian eigenvalues on $\Sigma$;
\item if $\partial\Sigma\ne\emptyset$, $\lim_{h\to 0^+}\lambda_j(\Omega_h)=\mu_j^D$, where $\{\mu_j^D\}_{j=1}^{\infty}$ are the Dirichlet Laplacian eigenvalues on $\Sigma$.
\end{enumerate}
\end{thm}
An immediate consequence of Theorem \ref{main_convergence} is that we can always find examples of domains $\Omega$ with any number of arbitrarily small eigenvalues $\la_j(\Omega)$, in the class of domains with prescribed volume $|\Omega|$ (see \cite{anne_tak} for a recent related result).

\begin{cor}\label{cor_small}
For any $N\in\mathbb N$  and $\epsilon>0$ there exists a domain $\Omega$  with $|\Omega|=1$ and $\lambda_j(\Omega)\leq\epsilon$ for all $j=1,...,N$. Moreover, the domain $\Omega$ can be chosen to be homeomorphic to a ball.
\end{cor}
\begin{proof}
Let us first prove the result in the case of $\Omega$ not homeomorphic to a ball. Let $\Sigma$ be a Cheeger's dumbbell \cite[p. 79]{chavel} with $N-1$ thin passages such that $\mu_1,...,\mu_N<\frac{\epsilon}{2}$ (in particular, $\mu_1=0$).
Choose a small enough $h>0$ guaranteeing that $\lambda_j(\Omega_h)<\epsilon$ for $j=1,...N$. Now, for $h$ small, $|\Omega_h|\approx |\Sigma|h$ and we can possibly reduce $h$ in such a way that $|\Omega_h|<1$. Let $\Omega:=\frac{\Omega_h}{|\Omega_h|^{1/3}}$, so that $|\Omega|=1$. Then $\lambda_j(\Omega)=|\Omega_h|^{2/3}\lambda_j(\Omega_h)<\epsilon$ for all $j=1,...,N$.\\
To produce a domain homeomorphic to a ball, replace $\Sigma$ in the above construction with $\Sigma_{\delta}:=\Sigma\setminus B_{\delta}$, where $B_{\delta}\subset\Sigma$ is a small geodesic disk with $\delta>0$, chosen in such a way that the Dirichlet eigenvalues $\mu_j^D$ in $\Sigma_{\delta}$ satisfy $\mu_1^D,...,\mu_N^D<\frac{2\epsilon}{3}$. This is possible because the Dirichlet spectrum of $\Sigma_{\delta}$ converges to the spectrum of the Laplacian on $\Sigma$ as $\delta\to 0^+$, see e.g., \cite{courtois_holes}. The rest of the construction is as in the previous part of the proof. It is sufficient to note that a thin tube around $\Sigma_{\delta}$ (or, equivalently, $\Sigma_{\delta}\times(0,h)$) is homeomorphic to a ball.
\end{proof}

Corollary \ref{cor_small} implies that a Faber-Krahn inequality cannot hold for the first Maxwell's eigenvalue, nor for other functions of the eigenvalues like the sum or the product (or other elementary symmetric functions) of the first $N$ eigenvalues, as already highlighted in \cite{lam_zacca} (see also \cite{savo_convex}). 

\medskip

Combining Corollary \ref{cor_small} with the examples of convex domains of fixed volume and arbitrarily large first eigenvalue (see e.g., \cite{lam_zacca,savo_convex}, or simply take $(0,\delta)\times(0,\delta)\times(0,1/\delta^2)$ which has large first eigenvalue when $\delta$ is small by Theorem \ref{thm_main_prod}) we conclude that for any $N\in\mathbb N$ and any $\epsilon, M>0$, there exist domains $\Omega,\omega$ homeomorphic to a ball with $|\Omega|=|\omega|=1$, such that $\lambda_j(\Omega)>M$ and $\lambda_j(\omega)<\epsilon$, for all $j=1,...,N$.

\medskip

To prove Theorem \ref{main_convergence}, it is convenient to change perspective and interpret problem \eqref{eq:curlcurl} as an eigenvalue problem for the Hodge Laplacian acting on co-closed differential $1$-forms with relative boundary conditions on a Riemannian manifold $(M,g)$ (see also problem \eqref{hodge_coclosed}):
\begin{equation}\label{hodge_INTRO}
\begin{cases}
\Delta u=\lambda u\,, & {\rm in\ }M\\
\delta u=0\,, & {\rm in\ } M\\
i^*u=0\,, &{\rm on\ }\partial M,
\end{cases}
\end{equation}
where now $u$ is a differential $1$-form on $M$, $\Delta=d\delta+\delta d$ is the Hodge Laplacian associated with the metric $g$ acting on differential forms,  $d$ is the exterior derivative, $\delta$ is the codifferential associated with the metric $g$ and $i:\partial M\to M$ is the canonical inclusion. We refer to Section \ref{sec:geom_pre} for more details. When $M$ is a bounded domain in $\mathbb R^3$ and the metric is the Euclidean one, problems \eqref{eq:curlcurl} restricted to divergence-free fields and \eqref{hodge_INTRO} coincide (under the canonical identification of vector fields and $1$-forms). 

It is worth recalling the following fact. For small $h>0$, the domain $\Omega_h$ with the Euclidean metric is quasi-isometric to the manifold $M=\Sigma\times(0,h)$ with the product metric $g_p=g_{\Sigma}\times dt^2$, where $g_{\Sigma}$ is the induced metric on $\Sigma$ from the ambient Euclidean space. Using the product structure of the metric, we are able to explicitly describe  all the eigenvalues of Problem \eqref{hodge_INTRO} in $M=\Sigma\times(0,h)$ and the associated eigenfunctions. Concerning the eigenvalues, we have the following (see Theorems \ref{thm_main_prod} and \ref{thm_main_prod_2})

\begin{thm}\label{thm_main_prod_INTRO}
Let $M=\Sigma\times(0,h)$, $h>0$, $(\Sigma,g_{\Sigma})$ be a compact Riemannian surface (without boundary) and $g_p=g_{\Sigma}+dt^2$ be the product metric on $M$. Then the spectrum of \eqref{hodge_INTRO} on $(M,g_p)$ is given by the union of the following four families:
\begin{enumerate}[i)]
\item $\mu_k+\eta_j(h)$, $k\geq 2$, $j\geq 1$;
\item $\mu_k+d_j(h)$, $k\geq 2$, $j\geq 1$;
\item $d_j(h)$, $j\geq 1$, each repeated $2\gamma$ times;
\item $0$ with multiplicity $1$.
\end{enumerate}
Here $\mu_k$ are the eigenvalues of the Laplacian on $(\Sigma,g_{\Sigma})$ (with multiplicities), $\eta_j(h),d_j(h)$ are the Neumann and Dirichlet eigenvalues on $(0,h)$, and $\gamma$ is the genus of the surface. 
\end{thm}

\begin{thm}\label{thm_main_prod_2_INTRO}
Let $M=\Sigma\times(0,h)$, $h>0$, $(\Sigma,g_{\Sigma})$ be a compact Riemannian surface with non-empty boundary $\partial\Sigma$ and $g_p=g_{\Sigma}+dt^2$ be the product metric on $M$. Then the spectrum of \eqref{hodge_INTRO} on $(M,g_p)$ is given by the union the following three sequences:
\begin{enumerate}[i)]
\item $\mu_k^D+\eta_j(h)$, $k,j\geq 1$;
\item $\mu_k^N+d_j(h)$, $k \geq 2$, $j\geq 1$;
\item $d_j(h)$, $j\geq 1$, each repeated $2\gamma+b$ times.
\end{enumerate}
Here $\mu_k^D,\mu_k^N$ are the eigenvalues of the Laplacian on $(\Sigma,g_{\Sigma})$ with Dirichlet and Neumann boundary conditions, respectively (with multiplicities), $\eta_j(h),d_j(h)$ are the Neumann and Dirichlet eigenvalues on $(0,h)$,  $\gamma$ is the genus of the surface, and $b+1$ is the number of connected components of $\partial\Sigma$.
\end{thm}
We note that in both cases, as $h\to 0^+$, all eigenvalues diverge to $+\infty$ except $\mu_k+\eta_1(h)$ ($k\geq 1$) if $\partial\Sigma=\emptyset$ and $\mu_k^D+\eta_1(h)$ ($k\geq 1$) if $\partial\Sigma\ne\emptyset$. In fact $\eta_1(h)=0$ for all $h$. The quasi-isometry between $(M, g_p)$ and $(M, g_E)$,
where $g_p$ is the product metric and $g_E$ is the Euclidean metric, finally allows us to conclude that the corresponding eigenvalues are at most at distance $Ch$ from each other, concluding therefore the proof of Theorem \ref{main_convergence}.


\medskip

Theorems \ref{thm_main_prod_INTRO} and \ref{thm_main_prod_2_INTRO} should be compared with the case of ``flat''  product domains of $\mathbb R^3$ of the form $\Omega = \omega \times I$, where $\omega \subset \R^2$ and $I \subset \R$. For such domains, it is well-known that the eigenvalues $\la_j(\Omega)$ belong to three different families (in the following list we keep the notation of \cite{coda}): 
\begin{enumerate}[i)]
\item the TE-modes, $\la^{\rm TE}_{jm}(\Omega) = \la_j(-\Delta_\omega^{\rm neu}) + \la_m(-\Delta_I^{\rm dir}) $, $j\geq 2$, $m \geq 1$;
\item the TM-modes, $\la^{\rm TM}_{jm}(\Omega) = \la_j(-\Delta_\omega^{\rm dir}) + \la_m(-\Delta_I^{\rm neu}) $, $j\geq 1$, $m \geq 1$;
\item when $\omega$ is not simply connected, and $\p \omega$ has $D$ connected components, the TEM modes: $\la^{\rm TEM}_{dm}(\Omega) = \la_m(-\Delta_I^{\rm dir}) $, $1\leq d \leq D-1$, $m \geq 1$;
\end{enumerate}

Theorems \ref{thm_main_prod_INTRO} and \ref{thm_main_prod_2_INTRO} say that the TE-TM-TEM description of the eigenvalues given in \cite{coda} continues to hold in the Riemannian setting, that is, when we replace $\omega$ with a Riemannian surface $\Sigma$, therefore generalising the construction valid for straight cylinders to possibly curved ones. The only difference is that the Maxwell's eigenvalues will now be described in terms of the eigenvalues of the Laplacian on the surface  $\Sigma$ (with Dirichlet or Neumann boundary conditions on $\p \Sigma$ when $\p \Sigma \neq \emptyset$, as in the flat case).

From the description in \cite{coda}, we easily see that the limiting spectrum of the Maxwell's operator on the flat cylinder $\omega\times I$ as $|I| \to 0^+$ coincides with the Dirichlet spectrum of the Laplacian on $\omega$: this is a particular case of Theorem \ref{main_convergence}.

\medskip
Another implication of Theorem \ref{main_convergence} is that there is no spectral stability under singular domain perturbation when a change of topology is involved. More precisely, let $\Sigma_{\delta}=\omega\setminus B_{\delta}\subset\mathbb R^2$, where $\omega$ is a simply connected planar domain, $B_{\delta}$ is a disk of radius $\delta$ centered at some $x\in\omega$, $B_{\delta}\subset\omega$ for all $\delta>0$ sufficiently small. Let $\Omega_{h,\delta}=\Sigma_{\delta}\times(0,h)$. Let $h>0$ be fixed. Then the Maxwell's eigenvalues on $\Omega_{h,\delta}$ are just those given by Theorem \ref{thm_main_prod_2_INTRO}:
\begin{enumerate}[i)]
\item $\mu_k^D(\delta)+\frac{\pi^2(j-1)^2}{h^2}$, $k,j\geq 1$, where $\mu_k^D(\delta)$ are the Dirichlet eigenvalues on $\Sigma_{\delta}$;
\item $\mu_k^N(\delta)+\frac{\pi^2j^2}{h^2}$, $k\geq 2$, $j\geq 1$, where $\mu_k^N(\delta)$ are the Neumann eigenvalues on $\Sigma_{\delta}$;
\item $\frac{\pi^2j^2}{h^2}$, $j\geq 1$.
\end{enumerate}
When $\delta\to 0^+$, $\Omega_{h,\delta}$ converges (in the sense of Hausdorff convergence) to $\Omega_h=\omega\times(0,h)$. The Maxwell's spectrum on $\Omega_h$ is given by Theorem \ref{thm_main_prod_2_INTRO}; however, we note that, since $\omega$ is simply connected, we do not have the third family of eigenvalues:
\begin{enumerate}[i)]
\item $\mu_k^D+\frac{\pi^2(j-1)^2}{h^2}$, $k,j\geq 1$, where $\mu_k^D$ are the Dirichlet eigenvalues on $\omega=\Sigma_0$;
\item $\mu_k^N+\frac{\pi^2j^2}{h^2}$, $k\geq 2$, $j\geq 1$, where $\mu_k^N$ are the Neumann eigenvalues on $\omega=\Sigma_0$.
\end{enumerate}
The first two families of eigenvalues behave continuously in $\delta$  (this follows from the spectral stability of the Dirichlet and Neumann eigenvalues in Euclidean domains upon removal of a small ball). On the contrary, the eigenvalues of the third family clearly admit a limit as $\delta\to 0^+$ (they do not depend on $\delta$), but the limits are not Maxwell's eigenvalues in the limit domain.

Associated with the families of eigenvalues in Theorems \ref{thm_main_prod_INTRO} and \ref{thm_main_prod_2_INTRO}, there are families of eigenfunctions that we describe more explicitly in Theorems \ref{thm_main_prod} and \ref{thm_main_prod_2}. In these theorems the eigenfunctions are interpreted as eigenfunctions of the Hodge Laplacian (problem \ref{hodge_coclosed}), that is, they are $1$-forms. Finally, we prove that  the eigenfunctions in $\Omega_h$ and the limit eigenfunctions in $\Sigma$ converge in a suitable sense as $h\to 0^+$, see Theorem \ref{thm:eigenfunctions} and Corollary \ref{cor:eigenfunctions}.

\medskip

The analysis of eigenvalue problems for differential operators on thin domains is a classical topic that has experienced a noticeable growth in recent years, see e.g., \cite{arrieta_mpereira,ArrVil2,bor_freitas,BCL,grieser,krej_DN,krej_2}
and references therein. Our analysis was inspired by the well-known result in \cite{schatzman}: the Neumann eigenvalues of a thin tube around a closed embedded hypersurface in $\mathbb R^n$ converge to the eigenvalues of the Laplacian on the surface. Since then, the analysis of the behavior of the spectrum in the thin limit turned out to be useful in the study of many other spectral problems, e.g., the hot spot conjecture \cite{krej_tus}, the clamped plate equation \cite{buoso_ferr}, Navier-Stokes equations \cite{miura2,miura3,miura1}, quantum waveguides \cite{post_exner1,post_exner2,post_exner3,rub_scha}. It is worth mentioning the recent paper \cite{LotOB}, in which the authors perform an asymptotic analysis of the spectrum of the Dirac operator in tubes around hypersurfaces in $\mathbb R^n$. When $n=3$, these are the domains considered in the present paper. When the tube shrinks to the surface, they are able to identify the effective Schr\"odinger operator driving the dynamics, and write an asymptotic expansion of the eigenvalues. From the geometric point of view, the analysis of the spectrum of the Hodge Laplacian acting on $p$-forms on domains with thin parts has been considered by various authors, see e.g., \cite{anne_colb,anne_post,anne_tak}. However, also from the geometric point of view, we were not aware of a result in the spirit of \cite{schatzman} for $p$-forms. This is the main motivation of the present note, which focuses on $p=1$ due to the relation of the problem on forms with the Maxwell's problem. We finally remark that in the last ten years there has been an upsurge of interest in the connection between the spectrum of the Maxwell's operator and the underlying geometry, mainly in the Euclidean setting, see, for instance \cite{BFMT, ZaccArX, FM24, Fil2, Fil1}, where, for instance, the role of the topology of the domain in the spectral properties of the Maxwell's system prominently appears. 

\medskip

The present note is organised as follows. Section \ref{sec:geom_pre} contains a few geometric preliminaries and the description of the connection between problems \eqref{eq:curlcurl} and \eqref{hodge_INTRO}. In Section \ref{sec:Mspectrum_hodge} we prove Theorems \ref{thm_main_prod_INTRO} and \ref{thm_main_prod_2_INTRO}, namely, we describe the Maxwell's eigenvalues and the eigenfunction of the product manifold $(M,g_p)=(\Sigma\times(0,L),g_p)$.  In Section \ref{sec:main} we prove our main Theorem \ref{main_convergence}. In Section \ref{sec:conv:eig} we establish a convergence result for eigenfunctions. Finally, we remark that in this note we identify the Maxwell's problem with the eigenvalue problem for the Hodge Laplacian restricted on co-closed $1$-forms with relative conditions.  For the reader's convenience, in Section \ref{sec:full} we describe the spectrum of the full Hodge Laplacian with relative boundary conditions in the product manifold $(M,g_p)$.

\smallskip

We conclude this section by underlining that this article is purposely written to be accessible to both mathematical analysts and differential/spectral geometers, and therefore it may contain details that are usually omitted in a research article.

\section{Geometric preliminaries and the interpretation of problem \eqref{eq:curlcurl} as an eigenvalue problem for the Hodge Laplacian}\label{sec:geom_pre}

\subsection{Notation and functional spaces}\label{sub:func:eucl}
Let us first describe the functional spaces that are involved in the analysis of the $\curl \curl$ equation in the case of a Lipschitz bounded domain $\Omega \subset \R^3$. The ambient Hilbert space will be $L^2(\Omega)^3$. Let $\nabla H^1_0(\Omega):=\{\nabla u:u\in H^1_0(\Omega)\}$ and $H(\Div0,\Omega):=\{E\in L^2(\Omega)^3:\Div E=0\}$. We recall that we have the classical Helmholtz decomposition 
\begin{equation}\label{eq:Helmholtz}
L^2(\Omega)^3 = \nabla H^1_0(\Omega) \oplus H(\Div0, \Omega).
\end{equation}
Let us also define the space
\[
H_0(\curl, \Omega) = \{ u \in L^2(\Omega)^3 : \curl u \in L^2(\Omega)^3, \: \nu \times u|_{\p \Omega} = 0\, \},
\]
and similarly
\[
H(\Div, \Omega) = \{ u \in L^2(\Omega)^3 : \Div u \in L^2(\Omega) \}.
\]
In view of Weber's compactness result (see \cite{weber}), the space
\[
X_N(\Omega) := H_0(\curl, \Omega) \cap H(\Div, \Omega)
\]
is compactly embedded in $L^2(\Omega)^3$. Let us now consider the weak formulation of equation \eqref{eq:curlcurl}, which reads
\begin{equation} \label{eq:curlcurlweak}
\int_\Omega \langle\curl E, \curl H\rangle = \la \int_\Omega \langle E , H\rangle
\end{equation}
for all $H \in H_0(\curl, \Omega)$. One sees immediately that, if $\la = 0$, any $E=\nabla u$, $u\in H^1_0(\Omega)$ is a solution of \eqref{eq:curlcurlweak}. In fact, the space $\nabla H^1_0(\Omega):=\{\nabla u:u\in H^1_0(\Omega)\}$ is contained in the kernel of the operator $\curl$. 

There are now two (equivalent) ways of studying the spectrum of this operator. Either we study it in the Hilbert space $L^2(\Omega)^3$, and then $\la = 0$ is a point of essential spectrum of the operator; or we restrict the Hilbert space to $H(\Div 0, \Omega) = \{E\in L^2(\Omega)^3:\Div E=0\}$, which corresponds to restrict the domain of the operator $\curl \curl$ to the orthogonal of $\nabla H_0^1(\Omega)$. We proceed with this second option.

Thus, in the Hilbert space $H(\Div 0, \Omega)$, which is endowed with the usual $L^2(\Omega)^3$-norm, we consider the  sesquilinear form
\[
Q(E,H) = \int_\Omega \langle\curl E, \curl H \rangle
\]
with domain $\dom(Q) = H_0(\curl, \Omega) \cap H(\Div0, \Omega)$, which is compactly embedded in $H(\Div 0, \Omega)$. By the second representation theorem, there exists a unique positive self-adjoint operator $T$ such that
\[
(T u, v) = Q(u,v),
\]
for all $u \in \dom(T) := \{ u \in \dom(Q): T^{1/2} u \in \dom(Q)\}$, $v \in \dom(Q)$. Moreover, the compact embedding of $\dom(Q)$ into the ambient Hilbert space $H(\Div 0, \Omega)$ implies that the resolvent $T^{-1}$ is compact as an operator in $H(\Div 0, \Omega)$. By standard spectral theory we deduce that the spectrum of $T$ coincides with its discrete spectrum and can be described by a sequence of non-negative, isolated eigenvalues of finite multiplicity
\[
0 \leq \la_1(\Omega) \leq \la_2(\Omega) \leq \dots \leq \la_j(\Omega) \leq \dots\nearrow+\infty,
\]
where the eigenvalues are repeated according to their multiplicity.

\subsection{Hodge Laplacian and geometric functional setting}\label{sub:hodge}
The weak formulation \eqref{eq:curlcurlweak} and the functional setting described in Subsection \ref{sub:func:eucl}, which apparently seem tied to Euclidean $3$-dimensional domains, can be translated in the context of general compact Riemannian manifolds.

Let $(M,g)$ be a compact, orientable, $n$-dimensional Riemannian manifold. Let $\Omega^p(M)$ denote the vector space of smooth differential $p$-forms on the differentiable manifold $M$.

By $d$ we denote the exterior derivative: $d:\Omega^p(M)\to\Omega^{p+1}(M)$, which is the ordinary differential of a function for $p=0$. For example, in $\mathbb R^3$ with Cartesian coordinates $(x,y,z)$, for a $0$-form $f=f(x,y,z)$ we have $df=\partial_xf dx+\partial_yfdy+\partial_yfdz$. For a $1$-form $f=f_1dx+f_2dy+f_3dz$ we have $df=(\partial_xf_2-\partial_yf_1)dx\wedge dy+(\partial_yf_3-\partial_zf_2)dy\wedge dz+(\partial_zf_1-\partial_xf_3)dx\wedge dz$, etc.

The metric $g$ allows us to define a Hodge-star  operator $\star:\Omega^p(M)\to\Omega^{n-p}(M)$. More concretely, we first note that the Riemannian metric induces a scalar product on the space of $p$-forms which we denote by $\langle\cdot,\cdot\rangle_g$. Then, for any $\omega\in\Omega^p(M)$, the Hodge-star operator $\star$ is defined by the following identity
$$
\phi\wedge(\star\omega)=\langle\phi,\omega\rangle_g dv_g\,,\ \ \ \forall\phi\in\Omega^p(M),
$$
where $dv_g$ is the volume $n$-form for the metric $g$. For example, in $\mathbb R^3$ with the Euclidean metric $g_E$ and the canonical basis $dx,dy,dz$ of $1$-forms, one has: $\star dx=dy\wedge dz$, $\star dy=dz\wedge dx$, $\star dz=dx\wedge dy$, $\star 1=dx\wedge dy\wedge dz=dv_{g_E}$, $\star(dx\wedge dy\wedge dz)=\star dv_{g_E}=1$, etc. In particular, for a $1$-form $f=f_1dx+f_2dy+f_3dz$, $\star df=(\partial_yf_3-\partial_zf_2)dx+(\partial_zf_1-\partial_xf_3)dy+(\partial_xf_2-\partial_yf_1)dz$ which can be identified with ${\rm curl}f$.

The Hodge $\star$ allows us to define a codifferential $\delta:\Omega^p(M)\to\Omega^{p-1}(M)$:
$$
\delta\omega=(-1)^{n(p+1)+1}\star d\star.
$$
For example, if $p=1$ we always have $\delta=-\star d \star$. For a $1$-form in $\mathbb R^3$, $f=f_1dx+f_2dy+f_3dz$, we have $\delta f=-\partial_xf_1-\partial_yf_2-\partial_zf_3=-{\rm div}f$ (for the Euclidean metric $g_E$). For $n=3$, $p=2$, $\delta=\star d\star$. Note that $\delta$ depends on the metric $g$.

Finally, we can define the Hodge Laplacian $\Delta:\Omega^p(M)\to\Omega^p(M)$ by
$$
\Delta\omega=(\delta d+d\delta)\omega.
$$
Note that $\Delta$ depends on the metric $g$. We will omit the dependence of $\delta$ and $\Delta$ on the metric $g$ when it is clear from the context,  otherwise we will write $\delta_g,\Delta_g$. Further note that in $\mathbb R^n$ we have $\delta f=0$, for any function (or, equivalently, $0$-form) $f$, hence $-\Delta f=\delta d f=-{\rm div}\nabla f$ is the usual Laplacian on functions. In $\mathbb R^3$, given a $1$- form $f=f_1dx+f_2dy+f_3dz$, the Hodge Laplacian acts as
$$
\Delta f={\rm curl}\,{\rm curl}f-\nabla\,{\rm div}f.
$$
Hence, the eigenvalue equation in \eqref{eq:curlcurl} corresponds to $\Delta E=\lambda E$ when $E$ is {\it co-closed}, that is, when $\delta E=0$ ($-{\rm div}E=0$). Here, with abuse of notation, we have identified the vector field $E$ with its dual $1$-form. The boundary condition $\nu\times E=0$ on $\partial\Omega$ can be translated into $i^*E=0$ on $\partial\Omega$, where $i^*:\partial\Omega\to\Omega$ is the canonical inclusion. This condition  forces $E$ to be normal to the boundary. In the terminology of differential forms, $i^*E=0$ on $\partial\Omega$ is called {\it relative} condition. In the case of $0$-forms, the relative condition corresponds to Dirichlet boundary condition.

\medskip

 We now consider the functional setting for the Hodge Laplacian acting on forms. Having a scalar product induced by the metric on $\Omega^p(M)$, the definitions of $L^2$ spaces and Sobolev spaces extend naturally to differential $p$-forms: the space $L^2\Omega^p(M)$ is defined as the completion of $\Omega^p(M)$ with respect to the $L^2$-inner product on forms: $\int_{M}\langle \omega_1,\omega_2\rangle_gdv_g$. The Sobolev spaces $H^m\Omega^p(M)$, $m\in\mathbb N$ are defined analogously, through the natural connection $\nabla$ on $(M,g)$ induced by the Riemannian metric (which allows to differentiate  forms). It is then possible to define the analogous of the spaces $H_0({\rm curl},\Omega)$, $H({\rm div},\Omega)$, $H({\rm div} 0,\Omega)$ on $(M,g)$ for differential forms of any degree. More precisely, we have
\begin{multline*}
X_N(M,g):=\\
\{\omega\in L^2\Omega^p(M):d\omega\in L^2\Omega^{p+1}(M),\delta\omega\in L^2\Omega^{p-1}(M),{\rm\ and\ }i^*\omega=0{\rm \ on\ }\partial M\},
\end{multline*}
where $i:\partial M\to M$ is the canonical inclusion (if $\partial M=\emptyset$ the last condition in the definition of $X_N$ is empty). In the case of non-empty boundary, this is the space of differential $p$-forms in $L^2$ with differential and codifferential in $L^2$ and satisfying the {\it relative} boundary conditions (namely, they are normal to $\partial M$).

We also recall the fundamental Hodge-Morrey decomposition:
\begin{equation}\label{hodge_morrey}
L^2\Omega^p(M)=d\Omega^{p-1}_R(M)\oplus\delta\Omega^{p+1}(M)\oplus\mathcal H_R(M)
\end{equation}
where $\Omega^{p-1}_R(M):=\{\omega\in\Omega^{p-1}(M):i^*\omega=0{\rm\ on\ }\partial M\}$ and $\mathcal H_R(M):=\{\omega\in\Omega^p(M):d\omega=\delta\omega=0\,,i^*\omega=0{\rm \ on\ }\partial M\}$. By abuse of notation, the spaces in the decompositions denote the closure of the corresponding spaces of smooth $p$-forms with respect to the $L^2$ norm. In the case of a domain of $\mathbb R^3$ and $p=1$,  $L^2\Omega^1(\Omega)=L^2(\Omega)^3$ and $d\Omega^0(M)=\nabla H_0^1$ (up to the isomorphism identifying $1$-forms with vector fields), and therefore we recover the Helmholtz decomposition \eqref{eq:Helmholtz}. The analogous Hodge-Morrey decomposition holds for any Sobolev space $H^m\Omega^p(M)$, $m\geq 1$. 

We can now see that problem \eqref{eq:curlcurlweak} corresponds to
\begin{equation}\label{eq:weak:forms}
\int_M\langle d\omega,d\varphi\rangle_gdv_g=\lambda\int_M\langle \omega,\varphi\rangle_gdv_g
\end{equation}
for all $\varphi\in X_N(M,g)$ such that $\delta\varphi=0$ in $M$, in the unknown $\omega\in X_N(M,g)$, $\delta \omega=0$ and $\lambda\in\mathbb R$.

Finally, we recall Gaffney's inequality
\begin{equation}\label{gaffney}
\|\omega\|^2_{H^1\Omega^p(M)}\leq C_G\left(\|d\omega\|^2_{L^2\Omega^{p+1}(M)}+\|\delta \omega\|^2_{L^2\Omega^{p-1}(M)}+\|\omega\|^2_{L^2\Omega^p(M)}\right)
\end{equation}
which holds for $\omega\in X_N(M,g)$. It is clear now that all the discussion in Subsection \ref{sub:func:eucl} applies to this more general setting, since the embedding $H^1\Omega^p(M)\to L^2\Omega^p(M)$ is  compact. This means that problem \eqref{eq:weak:forms} is associated with a compact, self-adjoint operator in $L^2\Omega^p(M)$ with nonnegative, discrete spectrum. Gaffney's inequality was originally proved in \cite{gaf1,gaf2} for manifolds without boundary, and in \cite{fri_gaf} for manifolds with boundary. For manifolds with boundary, it holds also under very mild smoothness assumptions on $\partial M$, see \cite{mitrea}. More precisely, for Euclidean domains, a Lipschitz condition on the boundary and a uniform outer ball condition are sufficient to guarantee the validity of Gaffney's inequality.

For more details on Sobolev spaces of $p$-forms, Hodge-Morrey decomposition, Gaffney's inequality, and for other relevant results of functional analysis, we refer to \cite{schwarz_hodge}. For more information on the spectrum of the Hodge Laplacian on $p$-forms (with different types of boundary conditions) we refer, e.g., to \cite{guerini_savo,guerini_savo_0,raulot_savo,savo_DTN}.

\medskip

The geometric framework described above allows for a more general approach to eigenvalue problems of type \eqref{eq:curlcurl} and helps to avoid some technicalities that may arise from the explicit use of coordinates. It gives a  geometric meaning to the decomposition of the Maxwell's spectrum in the ${\rm TE}, {\rm TM}$ and ${\rm TEM}$ modes for domains $\omega\times I$ (where $\omega\subset\mathbb R^2$ and $I\subset\mathbb R$) contained in \cite{coda}. Moreover, even though it is not the purpose of the present paper, it can be applied in any dimension and any ambient Riemannian manifold. 

We refer to \cite{guerini_savo}  for a nice introduction to eigenvalue problems for the Laplacian on $p$-forms on manifolds with boundary. See also \cite{savo_convex} for a detailed exposition on the spectrum of the Laplacian on $p$-forms on convex Euclidean domains.

\medskip

\section{Hodge Laplacian spectrum on co-closed $1$-forms with relative conditions on $3$-dimensional product manifolds}\label{sec:Mspectrum_hodge}

Throughout this section, we shall denote by $M$ the following product manifold of dimension $3$: 
$$
M=\Sigma\times I,
$$
where $(\Sigma,g_{\Sigma})$ is a compact Riemannian surface, $I=(0,h)$ is an interval, and $h>0$. The surface $\Sigma$ is compact, connected, smooth, and it is allowed to have a smooth boundary (with any number of connected components) and a possibly non-trivial topology. By $g_{\Sigma}$ we denote a smooth Riemannian metric on $\Sigma$. By $x$ we shall denote a point of $\Sigma$ and by $t\in(0,h)$ the usual coordinate on $I$.

We consider the manifold $M$ endowed with the product metric $g_p=g_{\Sigma}\times dt^2$. With this choice, $(M,g_p)$ is a  product Riemannian manifold. Note that at any point $q=(x,t)\in M$, we have the canonical orthogonal decomposition $T_qM=T_x\Sigma\oplus T_tI$.

We consider the following eigenvalue problem
\begin{equation}\label{hodge_coclosed}
\begin{cases}
\Delta \omega=\lambda\omega\,, & {\rm in\ }M\\
i^*\omega=0\,, & {\rm on\ }\partial M
\end{cases}
\end{equation}
restricted to the space of {\it co-closed $1$-forms}, that is, $1$-forms $\omega$ verifying $\delta \omega=0$ . To be more precise, we should have written $\delta_{g_p}$ and $\Delta_{g_p}$ since the co-differential depends on the metric. We shall omit the subscript when its meaning is clear from the context. Here $i:\partial M\to M$ is the canonical inclusion. This problem is the eigenvalue problem for the Hodge Laplacian restricted to co-closed one-forms with {\it relative} boundary conditions, see \cite{guerini_savo}.

We now describe the eigenvalues and eigenfunctions of \eqref{hodge_coclosed} in the product manifold $M$. We start from the case $\partial \Sigma=\emptyset$.

\subsection{Eigenvalues and eigenfunctions when $\partial\Sigma=\emptyset$}

\begin{thm}\label{thm_main_prod}
Let $(M,g_p)$ be a product Riemannian manifold, with $M=\Sigma\times(0,h)$, $h>0$, $(\Sigma,g_{\Sigma})$ a compact Riemannian surface (without boundary) and $g_p$ the product metric. Let 
\[
\begin{split}
(\mu_k, w_k)_{k \geq 1} \: &\textup{be the eigencouples of the Laplacian on $(\Sigma, g_\Sigma)$}; \\
(\eta_j(h), v_j)_{j \geq 1} \: &\textup{be the eigencouples of the Neumann Laplacian on $(0,h)$};\\
(d_j(h), u_j)_{j \geq 1} \: &\textup{be the eigencouples of the Dirichlet Laplacian on $(0,h)$}.
\end{split}
\]
Then the spectrum of \eqref{hodge_coclosed} is given by the union the following four families:
\begin{enumerate}[i)]
\item $\mu_k+\eta_j(h)$, $k\geq 2$, $j\geq 1$;
\item $\mu_k+d_j(h)$, $k\geq 2$, $j\geq 1$;
\item $d_j(h)$, $j\geq 1$, each repeated $2\gamma$ times;
\item $0$ with multiplicity $1$.
\end{enumerate}
Here 
$\gamma$ is the genus of the surface. The corresponding eigenfunctions are given by
\begin{enumerate}[i)]
\item $F_{jk}(x,t)=\delta d(w_k(x)v_j(t)dt)$, $k\geq 2$, $j\geq 1$. 
\item $F_{jk}(x,t)=\star d(w_k(x)u_j(t)dt)$, $k\geq 2$, $j\geq 1$. 
\item $F_{jk}(x,t)=H_k(x)u_j(t)$, where $\{H_k\}_{k=1}^{2\gamma}$ is a basis of harmonic $1$-forms on $\Sigma$. 
\item $F(x,t)=dt$.
\end{enumerate}
\end{thm}

\begin{rem}
We can recognise the three families of modes described in \cite{coda} for cylinders and balls in $\mathbb R^3$. In particular, our first family of corresponds to the ${\rm TM}$ modes in \cite{coda}, the second family corresponds to the ${\rm TE}$ modes, while, if the topology is not trivial, our third family  corresponds to the ${\rm TEM}$ modes. 
\end{rem}

\begin{proof}
{\bf First family.}  We look for eigenfunctions $F$ of the form
$$
F=\delta d(fdt),
$$
where $f$ is a smooth function on $M$.
We have the following facts:
 \begin{itemize}
 \item[\textbullet] $F=d_{\Sigma}f_t+(\Delta_{\Sigma}f) dt$, where $d_{\Sigma}$ is the differential on $\Sigma$, and $f_t$ indicates the derivative of $f$ with respect to $t$;
 \item we have $\delta F=\delta^2d(fdt)=0$, hence $F$ is co-closed;
 \item[\textbullet] we have then
 $$
 \Delta F=\delta d F=d_{\Sigma}(\Delta_{\Sigma}f_t-f_{ttt})+(\Delta^2_{\Sigma}f-\Delta_{\Sigma}f_{tt})dt;
$$
\item[\textbullet] the boundary condition $i^*F=0$ reads $d_{\Sigma}f_t=0$.
 \end{itemize}
Here and in what follows, by $\Delta_{\Sigma}$ we denote the Laplacian on $\Sigma$ for the metric $g_{\Sigma}$ and by $\delta_{\Sigma}$ we denote the codifferential on $\Sigma$ for the metric $g_{\Sigma}$. Therefore, we need to solve
\begin{equation}\label{split0}
\begin{cases}
d_{\Sigma}(\Delta_{\Sigma}f_t-f_{ttt})+(\Delta^2_{\Sigma}f-\Delta_{\Sigma}f_{tt})dt=\lambda(d_{\Sigma}f_t+(\Delta_{\Sigma}f) dt)\,, & {\rm in\ }M\\
d_{\Sigma}f_t=0\,, & {\rm on\ }\Sigma\times\{0,h\}.
\end{cases}
\end{equation}
According to the separation-of-variables ansatz, we look for solutions of the form $f(x,t)=w(x)v(t)$. The equation preserves the separation of variables and therefore we obtain:
\begin{equation}\label{split1}
\begin{cases}
v\Delta^2_{\Sigma}w-v''\Delta_{\Sigma}w-\lambda v\Delta_{\Sigma}w=0\,, & {\rm in\ }M\\
v'(0)d_{\Sigma}w=v'(h)d_{\Sigma}w=0\,, & {\rm on\ }\Sigma\times\{0,h\}.
\end{cases}
\end{equation}
If $w={\rm const}$ on $\Sigma$ we would have $F=0$. Moreover, we can choose $w$ such that $\int_{\Sigma}w=0$, since the corresponding $F$ would not change. Hence, the boundary condition reads $v'(0)=v'(h)=0$.

\medskip

Note that if $f(x,t)=w(x)v(t)$ with $\int_{\Sigma}w=0$ solves \eqref{split1}, then it solves
\begin{equation}\label{split2}
d_{\Sigma}(\Delta_{\Sigma}f_t-f_{ttt})=\lambda d_{\Sigma}f_t\,,\ \ \  {\rm in\ }M
\end{equation}
thus it solves \eqref{split0}.

A standard ansatz to construct a solution is to require that there exist constants $c$ such that
$$
\begin{cases}
\Delta^2_{\Sigma}w+c\Delta_{\Sigma}w-\lambda\Delta_{\Sigma}w=0\,,  & {\rm in\ }\Sigma\,,\\
-v''(t)=cv(t)\,, & {\rm in\ }(0,h)\,,\\
v'(0)=v'(h)=0.
\end{cases}
$$
The first equation implies that $w$ is an eigenfunction of the Laplacian on $\Sigma$ with eigenvalue $\mu=\lambda-c>0$ , since we are in the subspace of $H^1(\Sigma)$ of functions $w$ with zero mean over $\Sigma$. 
Note that $\mu_1=0$, $\mu_2>0$. On the other hand $v$ must be a Neumann eigenfunction on $(0,h)$ with eigenvalue $c$. 
We conclude that
$$
\lambda=\mu_k+\eta_j(h)\,,\ \ \ k\geq 2,j\geq 1.
$$

\medskip

{\bf Second family.} Consider now the following functions:
$$
F=\star d(fdt)
$$
where $f$ is a smooth function in $M$. One checks that
$$
F=\star d_{\Sigma}f.
$$
It is immediate to check that $\delta F=-\delta^2\star(fdt)=0$, so $F$ is co-closed. Hence
$$
\Delta F=\delta dF=\star d_{\Sigma}(\Delta_{\Sigma}f-f_{tt}).
$$
Hence \eqref{hodge_coclosed} becomes
\begin{equation}
\begin{cases}
\star d_{\Sigma}(\Delta_{\Sigma}f-f_{tt})=\lambda\star d_{\Sigma}f\,, & {\rm in\ } M\,,\\
\star d_{\Sigma}f=0\,, & {\rm on\ } \Sigma\times\{0,h\}.
\end{cases}
\end{equation}
which is equivalent to
\begin{equation}\label{fam20}
\begin{cases}
 d_{\Sigma}(\Delta_{\Sigma}f-f_{tt})=\lambda d_{\Sigma}f\,, & {\rm in\ } M\,,\\
 d_{\Sigma}f=0\,, & {\rm on\ } \Sigma\times\{0,h\}.
\end{cases}
\end{equation}
Again, the separation-of-variables ansatz suggests to look for solutions of the form $f(x,t)=w(x)u(t)$. We obtain
$$
\begin{cases}
-u''d_{\Sigma}w+ u\,d_{\Sigma}\Delta_{\Sigma}w=\lambda ud_{\Sigma}w\,,  & {\rm in\ }M\\
u(0)d_{\Sigma}w=u(h)d_{\Sigma}w=0.
\end{cases}
$$
Note that we can take $w$ such that $\int_{\Sigma}w=0$. Indeed, adding to $f$ any function $\phi(t)$ which depends only on $t$, we would have $\star d(f dt+\phi dt)=\star d(fdt)+\star d(\phi dt)$ and $d(\phi(t)dt)=0$.

The same argument used for the first family shows that $w$ cannot be constant, otherwise $f=u(t)$ and hence $F=0$. This implies that necessarily $-u''(t)=c u(t)$ for some constant $c$ and $u(0)=u(h)=0$. This implies that $c=d_j(h)$ a Dirichlet eigenvalue on $(0,h)$, and that $w$ must solve
$$
d_{\Sigma}(\Delta_{\Sigma}w+d_j(h)-\lambda)=0,
$$
hence
$$
\Delta_{\Sigma}w-(\lambda-d_j(h))w=c
$$
for some constant $w$. But since $w$ has zero mean, integrating the previous in $\Sigma$ we get $c=0$. Therefore $w$ is an eigenfunction of $\Delta_{\Sigma}$ with eigenvalue $\lambda-d_j(h)>0$. We conclude that 
$$
\lambda=\mu_k+d_j(h)\,,\ \ \ k\geq 2,j\geq 1.
$$

\medskip

{\bf Third family.} This family arises when the topolgy of $\Sigma$ is not trivial. We set
$$
F=u(t)H(x)
$$
where $H$ is a harmonic $1$-form on $\Sigma$. The space of harmonic $1$-forms on a compact surface is finite dimensional and has dimension $2\gamma$, where $\gamma$ is the genus of the surface. Proceeding as above, we find that $u$ must be some Dirichlet eigenfunction on $(0,h)$ with eigenvalue $d_j(h)$, $j\geq 1$.  The resulting eigenvalues are
$$
\lambda=d_j(h)
$$
each one repeated $2\gamma$ times.

\medskip

{\bf Fourth family (eigenfunctions with zero eigenvalue)}. One eigenfunction is left out from this analysis, which is $dt$. We have that $dt$ is harmonic and satisfies the relative boundary conditions.  We remark that this is the only zero eigenvalue of the Hodge Laplacian on $1$-forms (not just restricted to co-closed forms) on $\Sigma\times(0,L)$. This is not surprising since the relative cohomology in degree $1$ for $\Sigma\times(0,L)$ has dimension $1$.

\medskip

{\bf Completeness of the eigenfunctions.}
To complete the proof of Theorem \ref{thm_main_prod}, it remains to establish that the 4 families of eigenfunctions found above span the whole of the Hilbert space $L^2$. Assume that there exists a co-closed $1$-form $\omega$, satisfying the relative boundary conditions, and orthogonal to all the eigenfunctions in the three families and to $dt$. We will show that $\omega$ must be zero. 

Note that $\omega$ is a $1$-form satisfying $\delta\omega=0$ in $M$ and $i^*\omega=0$ on $\partial M$. We start by testing with the eigenfunctions of the first family, which are of the form
$$
F=d_{\Sigma}f_t+(\Delta_{\Sigma}f)dt
$$
where $f=w_kv_j$, $k\geq 2$, $j\geq 1$, with $w_k$ and $v_j$ as in the statement of the Theorem. 
At any $(x,t)\in M$ we can write $\omega=\omega_{\Sigma}+\omega_t dt$, where $\langle\omega_{\Sigma},dt\rangle_g=0$ in $M$ and $\omega_t$ is a smooth function defined in $M$. The coefficients of $\omega_{\Sigma}$  are smooth functions on $M$. Then, for all $k\geq 2$, $j\geq 1$, we have

\begin{multline}
0
=\int_0^h\int_{\Sigma}\langle d_{\Sigma}f_t+(\Delta_{\Sigma}f)dt,\omega_{\Sigma}+\omega_t dt\rangle_g dv_{g_{\Sigma}}dt\\
=\int_0^h\int_{\Sigma}\langle d_{\Sigma}f_t,\omega_{\Sigma}\rangle_{g_{\Sigma}}+\Delta_{\Sigma}f\omega_t dv_{g_{\Sigma}}dt\\
=-\int_0^h\int_{\Sigma}f_t\delta_{\Sigma}\omega_{\Sigma}dv_{g_{\Sigma}}dt+\mu_k\int_0^h\int_{\Sigma}w_kv_j\omega_t dv_{g_{\Sigma}}dt\\=\int_0^h\int_{\Sigma}f_t(\omega_t)_tdv_{g_{\Sigma}}dt+\mu_k\int_0^h\int_{\Sigma}w_kv_j\omega_t dv_{g_{\Sigma}}dt.
\end{multline}
Here we have used that $0=\delta\omega=\delta_{\Sigma}\omega_{\Sigma}+(\omega_t)_t$, where  $(\omega_t)_t:=\partial_t\omega_t$ since the metric is in product form. Then 
\begin{multline}
0=\int_0^h\int_{\Sigma}f_t(\omega_t)_tdv_{g_{\Sigma}}dt+\mu_k\int_0^h\int_{\Sigma}w_kv_j\omega_t dv_{g_{\Sigma}}dt\\
=(n_j(h)+\mu_k)\int_0^h\int_{\Sigma}w_kv_j\omega_t dv_{g_{\Sigma}}dt.
\end{multline}
This implies that $\omega_t$ must be constant on $\Sigma$ for any $t$ (that is, $\omega_t$ is a  function of $t$ only). In particular, again from $\delta_{\Sigma}\omega_{\Sigma}=-(\omega_t)_t$ we get, from Stokes theorem
$$
0=\int_{\Sigma}\delta_{\Sigma}\omega_{\Sigma}=(\omega_t)_t|\Sigma|
$$
which means that $\omega_t$ is constant in $M$. The orthogonality with $dt$ implies that $\omega_t= 0$.

Now, let us consider the second family of eigenfunctions: 
$$
F=\star d_{\Sigma}f
$$
where $f=w_ku_j$, $k\geq 2$, $j\geq 1$, with $w_k$ and $u_j$ as in the statement of the Theorem.
We therefore have, for any $k\geq 2$, $j\geq 1$:
$$
0=\int_0^h\int_{\Sigma}\langle\star d_{\Sigma}f,\omega\rangle_gdv_{g_{\Sigma}}dt=\int_0^hu_j\left(\int_{\Sigma}\langle \star d_{\Sigma}w_k,\omega_{\Sigma}\rangle_{g_{\Sigma}}dv_{g_{\Sigma}}\right)dt.
$$
This means that a.e.
$$
\int_{\Sigma}\langle \star d_{\Sigma}w_k,\omega_{\Sigma}\rangle_{g_{\Sigma}}dv_{g_{\Sigma}}=0
$$
which implies
$$
\int_{\Sigma}w_k\star d_{\Sigma}\omega_{\Sigma}=0
$$
for all $k$. This means that $\star d_{\Sigma}\omega_{\Sigma}$ is constant on $\Sigma\times\{t\}$ for all $t$, namely, $\star d_{\Sigma}\omega_{\Sigma}=z(t)$. By Stokes theorem
$$
0=\int_{\Sigma}d_{\Sigma}\omega_{\Sigma}=z(t)|\Sigma|
$$
hence $z(t)=0$ and for any $t$, $\omega_{\Sigma}$ is closed on $\Sigma$.

\medskip

Now we consider the third family of eigenfunctions $H_k(x)u_j(t)$, $j\geq 1$, where $u_j$ are the Dirichlet eigenfunctions on $(0,h)$ and $\{H_k\}_{k=1}^{2g}$, is a basis of harmonic $1$-forms in $\Sigma$. Proceeding as above, we find that $\omega_{\Sigma}$ is co-exact for any $t$, which means that $\omega_{\Sigma}=-\star d_{\Sigma}\star \phi\, dv_{g_{\Sigma}}$ for some function $\phi$. Since we have proved that $\omega$ is closed, we conclude that $\Delta_{\Sigma}\phi=0$ for all $t$, hence $\phi$ is constant and $\omega_{\Sigma}=0$. Hence $\omega=c dt$ for some constant $c$. Since $\omega$ must be orthogonal to $dt$ (the last eigenfunction, associated with the eigenvalue $0$), necessarily $c=0$.
\end{proof}
We observe that, all the eigenvalues diverge to $+\infty$ as $h \to 0^+$, except for the family $\mu_k+\eta_1(h)=\mu_k$, $k\geq 2$, and the zero eigenvalue (which is $\mu_1$). 
We restate this result in the following corollary.

\begin{cor}\label{cor:prod:1}
Let $\lambda_j(h,g_p)$ denote the eigenvalues of the product manifold $(M,g_p)$ where $M=\Sigma\times(0,h)$. Then
$$
\lim_{L\to 0}\lambda_j(h,g_p)=\mu_j,
$$
where $\mu_j$ are the Laplacian eigenvalues of $(\Sigma,g_{\Sigma})$.
\end{cor}

\subsection{Eigenvalues and eigenfunctions when $\partial\Sigma\ne\emptyset$}

We consider now the case when $\Sigma$ has a boundary. The arguments remain essentially the same up to some minor changes that we now underline.

The first difference comes from the ansatz for the families of eigenfunctions. By separation of variables, the eigenfunctions must be in the form $f(x,t)=w(x)u(t)$, where the component $w(x)$ (which was an eigenfunction of the Laplacian on $\Sigma$ in Theorem \ref{thm_main_prod}) must now satisfy some boundary conditions on $\partial\Sigma$. The second difference that one can note with respect to the proof of the completeness of eigenfunctions in Theorem \ref{thm_main_prod} is that, when using the Stokes theorem on $\Sigma$, we have boundary terms; however, orthogonality with respect to harmonic 1-forms associated with boundary components of $\partial\Sigma$ allows to perform the same proof above, with almost no change. In fact, in the case where $ \partial \Sigma \neq \emptyset$, there are in general additional harmonic $1$-forms related to the connected components of $\partial\Sigma$.

\begin{thm}\label{thm_main_prod_2}
Let $(M,g_p)$ be a product Riemannian manifold, with $M=\Sigma\times(0,h)$, $h>0$, $(\Sigma,g_{\Sigma})$ a compact Riemannian surface with non-empty boundary, and $g_p$ the product metric. 
Let 
\[
\begin{split}
(\mu^D_k, w^D_k)_{k \geq 1} \: &\textup{be the eigencouples of the Dirichlet Laplacian on $(\Sigma, g_\Sigma)$}; \\
(\mu^N_k, w^N_k)_{k \geq 1} \: &\textup{be the eigencouples of the Neumann Laplacian on $(\Sigma, g_\Sigma)$}; \\
(\eta_j(h), v_j)_{j \geq 1} \: &\textup{be the eigencouples of the Neumann Laplacian on $(0,h)$};\\
(d_j(h), u_j)_{j \geq 1} \: &\textup{be the eigencouples of the Dirichlet Laplacian on $(0,h)$}.
\end{split}
\]
Then the spectrum of \eqref{hodge_coclosed} is given by the union the following three sequences:
\begin{enumerate}[i)]
\item $\mu_k^D+\eta_j(h)$, $k,j\geq 1$;
\item $\mu_k^N+d_j(h)$, $k \geq 2$, $j\geq 1$;
\item $d_j(h)$, $j\geq 1$, each repeated $2\gamma+b$ times.
\end{enumerate}
Here 
$\gamma$ is the genus of the surface and $b+1$ is the number of connected components of $\partial\Sigma$. The corresponding eigenfunctions are given by
\begin{enumerate}[i)]
\item $F_{jk}(x,t)=\delta d(w^D_k(x)v_j(t)dt)$, $k\geq 1$, $j\geq 1$.
\item $F_{jk}(x,t)=\star d(w^N_k(x)u_j(t)dt)$, $k\geq 2$, $j\geq 1$.
\item $F_{jk}(x,t)=H_k(x)u_j(t)$, where $\{H_k\}_{k=1}^{2\gamma+b}$ is a basis of harmonic $1$-forms on $\Sigma$ with relative conditions.
\end{enumerate}
\end{thm}

\begin{proof}
During this proof we will use the same notation for the families of eigenfunctions, as introduced in the proof of Theorem \ref{thm_main_prod}.

{\bf First family.} For the first family $F=d_{\Sigma}f_t+(\Delta_{\Sigma}f)dt$ the additional boundary condition reads $\Delta_{\Sigma}f=0$  and $d_{\partial\Sigma}f_t=0$ on on $\partial\Sigma\times(0,h)$, which, with the ansatz $f(x,t)=w(x)v(t)$ becomes
$$
\Delta_{\Sigma}w|_{\partial\Sigma}v(t)=0 {\ \ \ \rm and\ \ \ }d_{\partial\Sigma}w|_{\partial\Sigma}v'(t)=0\,,\ \ \ t\in(0,h).
$$
Recall that the equations to be solved are
$$
\begin{cases}
\Delta_{\Sigma}^2w+(c-\lambda)\Delta_{\Sigma}w=0\,, & {\rm in\ }\Sigma\\
-v''(t)=cv(t)\,, & {\rm in\ }(0,h)\\
\Delta_{\Sigma}w|_{\partial\Sigma}v(t)=0\,, & {\rm on\ }\partial\Sigma\times (0,h)\,,\\
d_{\partial\Sigma}w|_{\partial\Sigma}v'(t)=0\,, & {\rm on\ }\partial\Sigma\times (0,h)\\
d_{\Sigma}w v'(0)=d_{\Sigma}g v'(h)=0.
\end{cases}
$$
Note that $w\equiv{\rm const}$ on $\Sigma$ is not admissible, otherwise $F=0$; the last condition then reads $v'(0)=v'(h)=0$. Therefore $c=\eta_j(h)$ is a Neumann eigenvalue on $(0,h)$.
Setting $\phi=\Delta_{\Sigma}w$, we have that $\phi$ is an eigenfunction of the Dirichlet Laplacian on $\Sigma$ with eigenvalue $\mu^D=\lambda-\eta_j(h)$, $j\geq 1$. Now, for $j=1$ we have $\eta_1(h)=0$ and $v(t)={\rm const}$. Also, $\Delta_{\Sigma}(\phi+\mu^Dw)=0$, which means that $w=\phi+har$ is a linear combination of a Dirichlet eigenfunction $\phi$ and an harmonic function $har$. However, we may choose $w$ to be a Dirichlet eigenfunction, because for $j=0$ we have $F=\Delta_{\Sigma}(\phi+har)dt=(\Delta_{\Sigma}\phi) dt$. For $\eta_j(h)\ne\eta_1(h)=0$ the boundary condition imposes that $w$ is constant on any connected component of the boundary. In this case we have that
\begin{equation}\label{proof:harmonic_prob}
\begin{cases}
\Delta_{\Sigma}(\phi+\mu^Dw)=0\,, & {\rm in\ }\Sigma\,,\\
\phi+\mu^Dw={\rm const} \,, & {\rm on\ each\ connect.\ comp.\ of\ }\partial\Sigma.
\end{cases}
\end{equation}
If $\Sigma$ has just one connected component of the boundary, then the only solution is $\phi+\mu^Dw\equiv{\rm const}$; we may then choose $w$ to be an eigenfunction of the Dirichlet Laplacian on $\Sigma$ with eigenvalue $\mu^D$, since adding a constant does not change $F$.  If $\partial\Sigma$ has $b+1$ connected components, there exists $b$ independent solutions of \eqref{proof:harmonic_prob}, which contribute to the spectrum. Altogether we have identified a family of eigenvalues in the form
$$
\lambda=\mu_k^D+\eta_j(h)\,,\ \ \ k,j\geq 1
$$
and if we have $b+1$ connected components of $\partial\Omega$, we have $b$ copies of $\eta_2(h),\eta_3(h),\cdots$, which are the positive Neumann eigenvalues on $(0,h)$, or, equivalently, we have $b$ copies of $d_1(h),d_2(h),...$ which are the Dirichlet eigenvalues on $(0,h)$. We will list these eigenvalues corresponding to non-trivial topology of the boundary $\p \Sigma$ in the third family here below.

\medskip

{\bf Second family.} For the second family, the ansatz is $F=\star d_{\Sigma}f$. 
We recall that for $\Sigma$ without boundary, the boundary condition at $\Sigma\times\{0,h\}$ reads $\star d_{\Sigma}f=0$ which is equivalent to $d_{\Sigma}f=0$. \\
Considering now the $\p \Sigma \neq \emptyset$ case and setting $f(x,t)=w(x)u(t)$, the condition $d_{\Sigma}f=0$ amounts to $u(0)d_{\Sigma}w=u(h)d_{\Sigma}w=0$. On the lateral boundary $\partial\Sigma\times (0,h)$, the boundary condition implies that $u(t)\star d_{\Sigma}w$ must be normal, forcing $\partial_{\nu_{\Sigma}}w=0$ on $\partial\Sigma$. We end up with
$$
\begin{cases}
u''(t)d_{\Sigma}w-u(t)d_{\Sigma}\Delta_{\Sigma}w=-\lambda u(t)d_{\Sigma}w\,,  & {\rm in\ }M\\
u(0)d_{\Sigma}w=u(h)d_{\Sigma}w=0\,,\\
\partial_{\nu_{\Sigma}}w=0\,, & {\rm on\ }\partial\Sigma.
\end{cases}
$$
Following then the proof of Theorem \ref{thm_main_prod}, we conclude that all the eigenvalues are given by
$$
\lambda=\mu_k^N+d_j(h)\,,\ \ \ j=1,..., k=2,...
$$
with $\mu_k^N$ and $d_j(h)$ as in the statement of the Theorem.


\medskip

{\bf Third family.} Finally, for genus $\gamma\geq 1$, we have a third family given by $u_j(t)H_k(x)$, where 
$\{H_k(x)\}_{k=1}^{2\gamma}$ is such that $\{H_k\}_{k=1}^{2\gamma}\cup\{d_{\Sigma}\psi_k\}_{k=1}^{b}$ is a basis of  the harmonic $1$-forms on $\Sigma$ with relative boundary conditions. Here $\psi_k$ are defined by $\Delta\psi_k=0$ in $\Sigma$, $\psi_k={\rm const}\ne 0$ on $\partial\Sigma_k$ and $\psi_k= 0$ on $\partial\Sigma_i$, $i\ne k$, for $k=1,...,b$, where $\Sigma_i$, $i=1,...,b+1$ are the connected components on $\partial\Sigma$ (these are the functions found in the analysis of the first family). In fact, the first relative cohomology of a surface $\Sigma$ of genus $\gamma$ and $b+1$ boundary components is isomorphic to $\mathbb Z^{2\gamma+b}$.

Thus, we get the additional family of eigenvalues
$$
\lambda=d_j(h)\,,\ \ \ j\geq 1.
$$
each repeated $2\gamma$ times. We include in this family also the eigenfunctions defined though the functions $\psi_k$ found in the analysis of the first family.

The completeness of eigenfunctions is proved exactly as in Theorem \ref{thm_main_prod}.
\end{proof}

We observe that all the eigenvalues diverge to $+\infty$ as $h\to 0^+$, except for the family $\mu_k^D+\eta_1(h)=\mu_k^D$, $k\geq 2$. 
We state this results explicitly in the following corollary.

\begin{cor}\label{cor:prod:2}
Let $\lambda_j(h,g_p)$ denote the eigenvalues of the product manifold $(M,g_p)$ where $M=\Sigma\times(0,L)$. Then
$$
\lim_{L\to 0}\lambda_j(h,g_p)=\mu_j^D,
$$
where $\mu_j^D$ are the Dirichlet Laplacian eigenvalues of $(\Sigma,g_{\Sigma})$.
\end{cor}

\section{Convergence of eigenvalues on tubes around embedded surfaces: proof of Theorem \ref{main_convergence}}\label{sec:main}

We will now proceed to prove Theorem \ref{main_convergence}. The proof strategy will be the following: we relate the eigenvalues of the tube $\Omega_h$ with those of the product manifold $(M,g_p)$, where
$$
M=\Sigma\times(0,h)
$$
and $g_p$ is the product metric; then we show that for $h\to 0^+$ the two sequences of eigenvalues become arbitrarily close and conclude using Theorems \ref{thm_main_prod} and \ref{thm_main_prod_2}.

Throughout this section, $\Sigma$ is a smooth, compact, embedded orientable surface in $\mathbb R^3$, and  $g_{\Sigma}$ is the Riemannian metric on $\Sigma$ induced by the ambient Euclidean space. Let  $\Omega_h$ be defined by \eqref{tube}, namely, $\Omega_h:=\{x+t\nu:t\in(0,h)\,,x\in\Sigma\}$ where $\nu$ is a choice of the unit normal to $\Sigma$ vector field to $\sigma$. Since $\Sigma$ is embedded and smooth, there exists $h_0>0$ such that, for all $h\in(0,h_0)$, the parallel surface to $\Sigma$ at distance $h$ is smooth, and moreover $\Omega_h$ is diffeomorphic to $M:=\Sigma\times(0,h)$ through $\phi:\Sigma\times(0,h)\to\Omega_h$.  We may then use Fermi coordinates $(x,t)\in M=\Sigma\times (0,h)$. Moreover, the domain $\Omega_h$ with the Euclidean metric $g_E$ is isometric to $(M,\phi^*g_E)$, where $\phi^*g_E$ is the pull-back of the Euclidean metric through $\phi$. To abbreviate, we write
$$
g_F:=\phi^*g_E
$$
for the pull-back of the Euclidean metric on $M$.
One also sees that
$$
g_F=g_p\,,\ \ \ {\rm on\ }\Sigma\times\{0\},
$$
 where $g_p=g_{\Sigma}+dt^2$ is the product metric on $M=\Sigma\times (0,h)$. Since $M$ is compact, we deduce that, for all $h>0$ sufficiently small,
\begin{equation}\label{quasi_isom}
\|g_F-g_p\|_{C^2(M)}<Ch\,, \quad \|g_F^{-1}-g_p^{-1}\|_{C^2(M)}<Ch,
\end{equation}
uniformly in $M$. See e.g., \cite[\S 2]{brisson}. From now on, by $C$ we denote a positive constant which does not depend on $h$ and which may be re-defined line by line.

By $\lambda_j(h,g_p)$ we denote the eigenvalues of \eqref{hodge_coclosed} (restricted to co-closed forms) on $(M,g_p)$ and by $\lambda_j(h,g_F)$ the eigenvalues of \eqref{hodge_coclosed} (restricted to co-closed forms) on $(M,g_F)$. From the discussion above we have that
$$
\lambda_j(\Omega_h)=\lambda_j(h,g_F)
$$
where $\lambda_j(\Omega_h)$ are the eigenvalues of \eqref{eq:curlcurl}.

To prove Theorem \ref{main_convergence}, we compare the Rayleigh quotients defining $\lambda_j(h,g_F)$ and $\lambda_j(h,g_p)$. We have
$$
\lambda_j(h,g_F)=\inf_{\substack{U\subset V_F\\{\rm dim}\,U=j}}\max_{0\ne u\in U}\frac{\int_{M}|du|^2_{g_F}dv_{g_F}}{\int_{M}|u|^2_{g_F}dv_{g_F}}.
$$
where $V_F=\{u:\delta_{g_F} u=0{\rm\ in\ }M,i^*u=0{\rm\ on\ }\partial M\}$, that is, the subspace of $1$-forms that are co-closed and normal to $\partial M$. 
Analogously, we have that
\begin{equation}\label{minmax_p}
\lambda_j(h,g_p)=\inf_{\substack{U\subset V_p\\{\rm dim}\,U=j}}\max_{0\ne u\in U}\frac{\int_{M}|du|^2_{g_p}dv_{g_p}}{\int_{M}|u|^2_{g_p}dv_{g_p}}.
\end{equation}
where $V_p=\{u:\delta_{g_p} u=0{\rm\ in\ }M,i^*u=0{\rm\ on\ }\partial M\}$.
The two spaces $V_F$ and $V_p$ are not the same since the codifferential $\delta$ depends on the metric.

\medskip

However, given $u$ such that $\delta_{g_F}u\ne 0$, but $\delta_{g_p} u=0$, we can replace it by $u+dv$ where $v$ satisfies $\delta_{g_F} dv=-\delta_{g_F} u$, $v=0$ on $\partial M$. This is just
\begin{equation}\label{projection}
\begin{cases}
\Delta_Fv=-\delta_{g_F}u\,, & {\rm in\ }M\,,\\
v=0\,, & {\rm on\ }\partial M.
\end{cases}
\end{equation}
That is, we have projected $u$ on the subspace of co-closed $1$-forms for the metric $g_F$. But now
$$
d(u+dv)=du+d^2v=du.
$$
Summarizing, for any $u\in V_p$, there exists a unique $\tilde u\in V_F$, $\tilde u=u+dv$, with $\delta_{g_F}\tilde u=0$, $du=d\tilde u$, and $\tilde u$ is tangential for $g_F$. This last fact follows since for a solution $v$ of \eqref{projection}, $dv$ is tangential, and since a $1$-form $u$ is tangential for $g_F$ if and only if it is tangential for $g_p$. This is due to the structure of the metrics $g_F,g_p$ which, in local coordinates, assume the form of a block diagonal matrix splitting the components along $T\Sigma$  and $(0,h)$. Hence we can write

\begin{equation}\label{minmax_h2}
\lambda_j(h,g_F)=\inf_{\substack{U\subset V_p\\{\rm dim}\,U=j}}\max_{0\ne u\in U}\frac{\int_{M}|du|^2_{g_F}dv_{g_F}}{\int_{M}|u+dv|^2_{g_F}dv_{g_F}}.
\end{equation}
Now, we have
$$
\int_{M}|dv|_{g_F}^2=-\int_{M}\langle u,dv\rangle_{g_F}=-\int_{M}\langle u,dv\rangle_{g_F}+\int_{M}\langle u,dv\rangle_{g_p},
$$
where the first identity follows by multiplying  \eqref{projection} by $v$ and integrating by parts.
We have added the last summand which equals zero since $\delta_{g_p}u=0$. Hence
\begin{equation}\label{esti}
\int_{M}|dv|_{g_F}^2\leq\left|\int_{M}\langle u,dv\rangle_{g_F}-\int_{M}\langle u,dv\rangle_{g_p}\right|
\end{equation}
From \eqref{quasi_isom} and standard manipulations of the right-hand side of \eqref{esti} we deduce that there exists a constant $C$ not depending on $u,v$ such that
\begin{equation}\label{auxil_est}
\int_{M}|dv|^2_{g_F}\leq Ch\int_{M}|u|_{g_F}^2.
\end{equation}
Again, using \eqref{quasi_isom} and \eqref{auxil_est} we deduce the existence of a constant $C>0$ not depending on $u$ such that, for all $u\in V_p$,
\begin{multline}\label{almost_minmax}
(1-Ch)\frac{\int_{M}|du|^2_{g_p}dv_{g_p}}{\int_{M}|u|^2_{g_p}dv_{g_p}} \leq\frac{\int_{M}|du|^2_{g_F}dv_{g_F}}{\int_{M}|u+dv|^2_{g_F}dv_{g_F}} \leq (1+Ch) \frac{\int_{M}|du|^2_{g_p}dv_{g_p}}{\int_{M}|u|^2_{g_p}dv_{g_p}}.
\end{multline}
The result now follows from \eqref{minmax_p} and \eqref{minmax_h2}: for all $j$ we have
\begin{equation}\label{estimate_ev}
(1-Ch)\lambda_j(h,g_p)\leq\lambda_j(h,g_F)\leq(1+Ch)\lambda_j(h,g_p).
\end{equation}
This concludes the proof of Theorem \ref{main_convergence} since $\lambda_j(h,g_F)=\lambda_j(\Omega)$.

\section{Convergence of eigenfunctions}\label{sec:conv:eig}

In this section we establish a convergence result for eigenfunctions. Throughout this section
$$
M=\Sigma\times(0,h).
$$
We start by observing that, if we consider the eigenvalues $\Lambda_j(h,g_p)$ of the (full) Hodge Laplacian with relative boundary conditions on $(M,g_p)$, namely,
\begin{equation}\label{full_hodge_conv}
\begin{cases}
\Delta\omega=\Lambda\omega\,, & {\rm in\ }M\,,\\
i^*\omega=i^*\delta\omega=0\,, & {\rm on\ }\partial M,
\end{cases}
\end{equation}
then for all $j\in\mathbb N$, $\lim_{h\to 0^+}\Lambda_j(h,g_p)=\mu_j$, where $\mu_j,\mu_j^D$ are the eigenvalues of the Laplacian on $\Sigma$ (if $\partial M=\emptyset$). If $\partial\Sigma\ne\emptyset$, the same statement holds with $\mu_j^D$ as limit eigenvalue. Moreover, for any fixed $j\in\mathbb N$, there exists $h_j>0$ such that $\Lambda_j(h,g_p)=\mu_j$ for all $h<h_j$, and a basis of the corresponding eigenspace is given by $\{w_i(x)dt\}_{i=1}^{m_j}$, where $w_i$ are the eigenfunctions of the Laplacian on $\Sigma$ (with Dirichlet conditions if $\partial\Sigma\ne\emptyset$) associated with $\mu_j$. We refer to Appendix \ref{sec:full} for more details on the full Hodge Laplacian spectrum on product manifolds.

Moreover, $\delta w_j(x)dt=0$, hence the only eigenfunctions associated with Hodge Laplacian eigenvalues which admit a finite limit are co-closed. It is not difficult to see that these eigenfunctions correspond (up to scalar multiples) to the eigenfunctions $\delta d(w_j(x)dt)$ of the first family in Theorems \ref{thm_main_prod} and \ref{thm_main_prod_2}, which in turn are exactly those providing the co-closed eigenvalues that have a finite limit. 
Now note that $\delta d(w_jdt)=\Delta(w_jdt)=\mu_j w_j dt$. All these statements still hold when $\partial\Omega\ne\emptyset$, up to replacing $\mu_j$ with $\mu_j^D$.

Let us now consider problem \eqref{full_hodge_conv} on $(M,g_F)$.
The Hodge-Morrey decomposition \eqref{hodge_morrey} implies that the spectrum of \eqref{full_hodge_conv} on $(M,g_F)$
is given by the union of the co-closed spectrum, (i.e., the spectrum of \eqref{full_hodge_conv} restricted to co-closed $1$-forms), and the spectrum of \eqref{full_hodge_conv} restricted to the space of gradients of functions in $H^1_0(M)$. But, as $h\to 0^+$, all the eigenvalues of \eqref{full_hodge_conv} associated with eigenfunctions that are gradients of functions in $H^1_0(M)$ diverge to $+\infty$. Hence, if we denote by $\Lambda_j(h,g_F)$ the spectrum of \eqref{full_hodge_conv} for $(M,g_F)$, we have that for any fixed $j\in\mathbb N$ there exists $h_0>0$ such that $\Lambda_i(h,g_G)=\lambda_i(h,g_F)$ for all $i=1,...,j$. Recall that by $\lambda_j(h,g_F)$ and $\lambda_j(h,g_p)$ we have denoted the spectrum of \eqref{full_hodge_conv} restricted to co-closed forms.

We are now ready to prove the following

\begin{thm}\label{thm:eigenfunctions}
Let $\mu_j$ be an eigenvalue of the Laplacian on $\Sigma$ (with Dirichlet conditions if $\partial\Sigma\ne\emptyset$) and let $w$ be an associated $L^2$-normalized eigenfunction. Then there exists $h_0>0$ and an eigenfunction $u$ of problem \eqref{full_hodge_conv} on  $(M,g_F)$ associated with $\Lambda_j(h,g_F)$ satisfying
$$
\left\|h^{-1/2}wdt-u\right\|_{L^2\Omega^1(M)}\leq Ch
$$
for all $h<h_0$, where $C>0$ is a constant depending only on $\mu_j$ and $\Sigma$.
\end{thm}

In view of the previous discussion, from Theorem \ref{thm:eigenfunctions} we deduce the following corollary.
\begin{cor}\label{cor:eigenfunctions}
Let $\mu_j$ be an eigenvalue of the Laplacian on $\Sigma$ (with Dirichlet conditions if $\partial\Sigma\ne\emptyset$) and let $w$ be an associated $L^2$-normalized eigenfunction. Then there exists $h_0>0$ and an eigenfunction $u$ of problem \eqref{hodge_coclosed} associated with $\lambda_j(h,g_F)$ satisfying
$$
\left\|h^{-1/2}wdt-u\right\|_{L^2\Omega^1(M)}\leq Ch
$$
for all $h<h_0$,  where $C>0$ is a constant depending only on $\mu_j$ and $\Sigma$.
\end{cor}
This last corollary, translated in terms of solutions of Maxwell's problem \eqref{eq:curlcurl}, says that, for all $h>0$ sufficiently small, given an eigenfunction $w$ on the limit surface $\Sigma$ associated with a limit eigenvalue $\mu_j$ (or $\mu_j^D$ if $\Sigma$ has a boundary), there exists an eigenfunction $u$ of \eqref{eq:curlcurl} associated with $\lambda_j(\Omega_h)$ which is close in $L^2(\Omega)^3$ to the constant extension of $w$ in the normal direction to $\Sigma$.

\begin{proof}[Proof of Theorem \ref{thm:eigenfunctions}]

Let $\{w_i\}_{i=1}^{m_j}$ is a orthonormal basis of the eigenspace associated with $\mu_j$. 
First, note that there exists $h_0>0$, such that, for any $h<h_0$ an orthonormal basis of the  eigenspace associated with $\Lambda_j(h,g_p)$ is given by $\{h^{-1/2}w_idt\}_{i=1}^{m_j}$. Let $v=h^{-1/2}\sum_{i=1}^{m_j}a_iw_i$ with $\sum_{i=1}^{m_j}a_i^2=1$. Then $v$ is a $L^2$-normalized eigenfunction of $(M,g_p)$ associated with $\Lambda_j(h,g_p)$. Possibly choosing a smaller $h_0$, we may further assume that $\Lambda_{j-1}(h,g_p)<\Lambda_j(h,g_p)=\Lambda_{j+1}(h,g_p)=\cdots=\Lambda_{j+m_j-1}(h,g_p)<\Lambda_{j+m_j}(h,g_p)$ and $\Lambda_{j-1}(h,g_F)<\Lambda_j(h,g_F)=\Lambda_{j+1}(h,g_F)=\cdots=\Lambda_{j+m_j-1}(h,g_F)<\Lambda_{j+m_j}(h,g_F)$.

 From \eqref{estimate_ev} and from the fact that, for $h_0$ sufficiently small $\lambda_j(h,g_p)=\Lambda_j(h,g_p)$ and $\lambda_j(h,g_F)=\Lambda_j(h,g_F)$, we have that $|\Lambda_j(h,g_F)-\Lambda_j(h,g_p)|<Ch$, and the multiplicities of $\Lambda_j(h,g_F)$ and $\Lambda_j(h,g_p)$ coincide.  Since the eigenfunctions in the product metric are explicit (see Section 3), we can assume $v \in H^2\Omega^1(M)$ for the metric $g_p$ (and also for $g_F$), uniformly in $h$, as the eigenfunctions of the  Laplacian in $\Sigma$ (with Dirichlet conditions in the case $\Sigma$ has a boundary) are smooth. Let $u$ be an eigenfunction associated with $\Lambda_j(h,g_F)$ such that $\int_{M}\langle u_i^j,(u-v)\rangle_{g_F}dv_{g_F}=0$ for all $i=1,...,m_j$, where $u_i^j$ is a $L^2$-orthonormal basis of the eigenspace $\Theta_h$ associated with $\Lambda_j(h,g_F)$ (for the metric $g_F$). Note that if $h_0$ is sufficiently small, then $u\ne 0$.
 
Now, consider the following problem
\begin{equation}\label{eq:u-v}
(\Delta_{g_F}-\Lambda_j(h,g_F)) (u-v)=(\Delta_{g_F}-\Delta_{g_p})v-(\Lambda_j(h,g_F)-\Lambda_j(h,g_p))v
\end{equation}
in $(M,g_F)$.
On the boundary, we have that $i^*(u-v)=0$ (the outer unit normal is the same for both the metrics). Moreover, $i^*\delta_{g_F}(u-v)=i^*(\delta_{g_p}v-\delta_{g_F}v)=O(h)$ by \eqref{quasi_isom}; indeed, $\delta_{g_F} u = \delta_{g_p} v = 0$, since, upon choosing $h$ small enough, $u$ and $v$ must be eigenfunctions of the Hodge operator restricted to co-closed 1-forms. Then $u-v$ belongs to the orthogonal of the kernel of $\Delta_{g_F}-\Lambda_j(h,g_F)$, it is normal to the boundary and its divergence is small ($O(h)$) at the boundary. Then, if $F_h = (\Delta_{g_F}-\Delta_{g_p})v-(\lambda_j(h,g_F)-\lambda_j(h,g_p))v$, $q_h = i^*(\delta_{g_p}v-\delta_{g_F}v)$, the form $w = u-v$ solves the following inhomogeneous problem for the Hodge Laplacian with relative boundary conditions:
\[
\begin{cases}
(\Delta_{g_F}-\lambda_j(h)) w = F_h\,,  & {\rm in\ }M\,, \\
i^*w = 0,\,\,\, i^*\delta_{g_F} w = q_h, & {\rm on\ }\partial M.
\end{cases}
\]
Since $w \in \Theta_h^\perp$ for each fixed $h>0$, the Fredholm alternative implies that this problem is solvable provided that $F_h \in \Theta_h^\perp$, which holds true in view of the identity \eqref{eq:u-v}.

Interpreting this problem in weak form, we  obtain the $L^2\Omega^1(M)$ (for the metric $g_F$) a priori: estimate
\begin{multline*}
{\rm dist}(\sigma(\Delta_{g_F}) \setminus \{\Lambda_j(\Omega_h)\}, \Lambda_j(\Omega_h)) \norm{w}^2_{L^2\Omega^1(M)}\\ \leq C (\norm{F_h}^2_{L^2\Omega^1(M)} + \norm{q_h}^2_{L^2(\partial M)}),
\end{multline*}
where $C>0$ does not depend on $h$. Now note that
\[
\norm{F_h}_{L^2\Omega^1(M)} \leq C \big( \|g_p^{-1}-g_F^{-1}\|_{C^2(M)})\|v\|_{H^2\Omega^1(M)} + |\Lambda_j(\Omega_h) - \Lambda_j(h))| \norm{v}_{L^2\Omega^1(M)} \big)
\]
and 
\[
\norm{q_h}_{L^2(\p \Omega_h)} \leq C \, \|g^{-1}_p-g^{-1}_F\|_{C^2(M)}\, \norm{v}_{H^2(\Omega_h)}
\]
by the Trace inequality and the definition of $q_h$. Note that the constant $C>0$ needs to be possibly redefined, but it still does not depend on $h$. This concludes the proof in view of \eqref{quasi_isom} and of the convergence of the eigenvalues established in Section \ref{sec:main}.

\end{proof}

\begin{rem}
Concerning stronger notions of convergence, it is natural to expect that the resolvent operator of the Maxwell's system converges to that of the (Dirichlet) Laplacian on $\Sigma$. We believe that this is the case. However, the proof of this fact is beyond the scope of the paper which is devoted to the analysis of the eigenvalues and the consequences for the spectral geometry of this operator. Although the convergence in the case of the product metric could be achieved with standard arguments, the convergence in the Euclidean setting is more delicate. Indeed, the natural extension operator, taking a function $w$ on $\Sigma$ to a $1$-form $wdt$ in $M$, does not produce, in general, co-closed forms; and even more dramatically, the projection of $wdt$ on the space of co-closed forms may depend on the small parameter $h$.
\end{rem}

\appendix

\section{Full Hodge Laplacian spectrum with relative conditions on product manifolds}\label{sec:full}

We have seen that the Maxwell's eigenvalues coincide with the eigenvalues of the Hodge Laplacian with relative conditions restricted to the subspace of co-closed differential forms. On a compact Riemannian manifold $(M,g)$, the eigenvalue problem for the Hodge Laplacian with relative conditions acting on $p$ forms reads
\begin{equation}\label{eq:full:hodge}
    \begin{cases}
\Delta\omega=\lambda\omega\,, & {\rm in\ }M\\
i^*\omega=i^*\delta\omega=0\,, & {\rm on\ }\partial M.
    \end{cases}
\end{equation}
When $(M,g)$ is an Euclidean domain $\Omega$, identifying vectors and $1$-form, problem \eqref{eq:full:hodge} reads
\begin{equation}\label{eq:full:R3}
    \begin{cases}
{\rm curl}\,{\rm curl}E-\nabla{\rm div}E=\lambda E\,, & {\rm in\ }\Omega\,,\\
\nu\times E={\rm div}E=0\,, & {\rm on\ }\partial \Omega.
    \end{cases}
\end{equation}

It is well-known that if $p+q=r$, and if $\alpha$ is a $p$-form and $\beta$ is a $q$-form, then the Laplacian acting on the $r$-form $\alpha\wedge\beta$ on a product manifold $M\times N$ splits as follows (K\"unneth formula):
$$
\Delta(\alpha\wedge\beta)=\alpha\wedge\Delta\beta+\alpha\wedge\Delta\beta.
$$
In the case of closed manifolds $M,N$, the Hodge Laplacian spectrum of $r$ forms is then given by summing $p$-eigenvalues of $M$ and $q$-eigenvalues of $N$, for any choices of $p,q$ such that $p+q=r$. In the case of $1$ forms, the spectrum is then given by sums of Laplacian eigenvalues of $M$ and Hodge Laplacian eigenvalues on $1$-forms on $N$, and sums  of Laplacian eigenvalues of $N$ and Hodge Laplacian eigenvalues on $1$-forms on $M$.

We compute now the spectrum of \eqref{eq:full:hodge} for product Riemannian manifolds and $1$-forms, and in particular, for $M=\Sigma\times(0,h)$ and $g=g_p$, where $(\Sigma,g_{\Sigma})$ is a Riemannian surface and $g_p=g_{\Sigma}+dt^2$ is the product metric. Assume first that $\Sigma$ has no boundary. Consider the family
$$
u_j(t)\omega_k(x)\,,\ \ \ j,k\geq 1,
$$
where $u_j$ are the Dirichlet eigenfunctions of $(0,h)$ with eigenvalues $d_j(h)$ and $\omega_k$ are the eigenfunctions of the Hodge Laplacian on $1$-forms on $\Sigma$  with eigenvalues $\mu^1_k$. One checks that these are  eigenfunctions of the Hodge Laplacian on $\Sigma\times(0,h) $ with relative boundary conditions. The corresponding eigenvalues are given by $d_j(h)+\mu_k^1$, $j,k\geq 1$.
Consider now the family
$$
w_k(x)v_j(t)dt\,,\ \ \ j,k\geq 1
$$
where $v_j$ are the Neumann eigenfunctions of $(0,h)$ with eigenvalues $\eta_j(h)$ and $w_k$ are the eigenfunctions of the  Laplacian (on functions) on $\Sigma$ with eigenvalues $\mu_k$. One checks that these are  eigenfunctions of the Hodge Laplacian on $\Sigma\times(0,h) $ with relative boundary conditions. The corresponding eigenvalues are given by $\eta_j(h)+\mu_k$, $j,k\geq 1$. These two families exhaust the entire spectrum. This can be done as in the proofs of Theorems \ref{thm_main_prod} and \ref{thm_main_prod_2} (see also \cite{colbois_p}).

 If $\Sigma$ has a boundary, it is possible to argue in a similar way. Namely, one just needs to choose eigenfunctions in the form $u_j(t)\omega_k(x)$ with $u_j$ Dirichlet eigenfunctions on $(0,h)$ and $\omega_k$ $1$-forms with relative conditions on $\Sigma$, or in the form $w_k(x)v_j(t)dt$ with $w_k(x)$ Dirichlet eigenfunctions on $\Sigma$ and $v_j(t)$ Neumann eigenfunctions on $(0,h)$. 

\medskip
 
Note that the Hodge Laplacian spectrum restricted to co-closed $1$-forms (Maxwell's spectrum) is a subset of the Hodge Laplacian spectrum. However, it is not immediate to recognize the explicit form of the co-closed eigenfunctions as in Theorems \ref{thm_main_prod} and \ref{thm_main_prod_2}. Nevertheless, if we denote by $\Lambda_j(h,g_p)$ the eigenvalues of the (full) Hodge Laplacian on $(M,g_p)$ with relative conditions, then, for all $j\in\mathbb N$,
$$
\lim_{h\to 0^+}\Lambda_j(h,g_p)=\mu_j,
$$
where $\mu_j$ are the eigenvalues of the Laplacian on $(\Sigma,g_{\Sigma})$ (with Dirichlet conditions if $\partial\Sigma\ne\emptyset$). More precisely, for each fixed $j$, there exists $h_j$ sufficiently small such that $\Lambda_j(h,g_p)=\mu_j$ for all $h<h_j$ and the eigenfunctions associated with $\Lambda_j(h,g_p)$ are given by $w_j(x)dt$, with $w_j(x)$ the eigenfunctions of the Laplacian on $\Sigma$ associated with $\mu_j$ (with Dirichlet conditions if $\partial\Sigma\ne\emptyset$). This analysis is actually sufficient to prove Corollaries \ref{cor:prod:1} and \ref{cor:prod:2}, and then Theorem \ref{main_convergence}. However, our method of proof also provides an explicit description of the eigenfunctions on the product manifold, see Theorems \ref{thm_main_prod} and \ref{thm_main_prod_2}.

Finally, we mention that there is a dual set of boundary conditions for the eigenvalue problem for the  Hodge Laplacian. More precisely, we can consider the eigenvalue problem for the Hodge Laplacian with {\it absolute} boundary conditions:
\begin{equation}\label{hodge_abs}
\begin{cases}
\Delta\omega=\lambda\omega\,, & {\rm in\ } M\,,\\
i^*\iota_{\nu}\omega=i^*\iota_{\nu}d\omega=0\,, & {\rm on\ }\partial M\,,
\end{cases}
\end{equation}
where $\iota$ denotes the interior multiplication of differential forms. For $0$-forms (functions), this simply reduces to the Neumann boundary conditions. By the Hodge $\star$ isomorphism, which exchanges the two boundary conditions, we have that
$$
\lambda_{j,p}^R(M,g)=\lambda_{j,n-p}^A(M,g)
$$
where in the above formula, $\lambda_{j,p}^R(M,g)$ and $\lambda_{j,p}^A(M,g)$ denote the eigenvalues of the Hodge Laplacian acting on $p$-forms on the $n$-dimensional Riemannian manifold $(M,g)$ with relative and absolute boundary conditions, respectively. In particular, due to the equality $\la_{j,1}^R = \la^A_{j,2}$ when $n=3$, we conclude that our analysis of the Hodge Laplacian eigenvalues with relative boundary conditions for 1-forms also yields the same results for 2-forms with absolute boundary conditions.  A nice introduction to the Hodge Laplacian on domains with boundary can be found e.g., in \cite{guerini_savo}.

\section*{Acknowledgments} The authors are very thankful to Prof. Alessandro Savo for the numerous discussions on the topic and for reading a preliminary version of the paper, providing helpful advice, and to Prof. Francesco Bei for his helpful advice.

\bibliography{bibliography.bib}
\bibliographystyle{abbrv}
\end{document}